\documentclass[11pt,reqno]{amsart}

\usepackage[T1]{fontenc}
\usepackage[utf8]{inputenc}
\usepackage{lmodern}
\usepackage{microtype}
\usepackage{amsmath,amssymb,mathtools,mathrsfs}
\usepackage{enumitem}
\usepackage[colorlinks=true,linkcolor=blue,citecolor=blue,urlcolor=blue]{hyperref}
\usepackage[nameinlink,capitalise,noabbrev]{cleveref}

\numberwithin{equation}{section}

\newtheorem{theorem}{Theorem}[section]
\newtheorem{proposition}[theorem]{Proposition}
\newtheorem{lemma}[theorem]{Lemma}
\newtheorem{corollary}[theorem]{Corollary}

\theoremstyle{definition}
\newtheorem{definition}[theorem]{Definition}

\theoremstyle{remark}
\newtheorem{remark}[theorem]{Remark}

\DeclareMathOperator{\Irr}{Irr}
\DeclareMathOperator{\Hom}{Hom}
\DeclareMathOperator{\Span}{span}
\DeclareMathOperator{\length}{length}
\DeclareMathOperator{\ch}{ch}
\DeclareMathOperator{\Stab}{Stab}
\DeclareMathOperator{\rank}{rank}

\newcommand{\C}{\mathbb C}
\newcommand{\Z}{\mathbb Z}
\newcommand{\kfield}{\Bbbk}
\newcommand{\cO}{\mathcal O}
\newcommand{\cH}{\mathcal H}
\newcommand{\cC}{\mathcal C}
\newcommand{\cQ}{\mathcal Q}
\newcommand{\cX}{\mathcal X}
\newcommand{\aff}{\mathrm{aff}}
\newcommand{\fl}{\mathrm{fl}}
\newcommand{\gr}{\mathrm{gr}}
\newcommand{\sgn}{\mathrm{sgn}}
\newcommand{\nch}{\operatorname{nch}}
\newcommand{\snch}{\operatorname{snch}}

\newcommand{\whge}{\widehat{\mathfrak g}^{\,e}}
\newcommand{\whh}{\widehat{\mathfrak h}}
\newcommand{\whW}{\widehat W}
\newcommand{\whrho}{\widehat\rho}
\newcommand{\whR}{\widehat R}
\newcommand{\CellL}{\mathbf c^{L}}
\newcommand{\CellR}{\mathbf c^{R}}
\newcommand{\dotact}{\mathbin{\circ}}
\newcommand{\dual}{\vee}
\newcommand{\eps}{\varepsilon}

\newcommand{\posC}{\underline C}
\newcommand{\negC}{C^{-}}
\newcommand{\longhookrightarrow}{\lhook\joinrel\longrightarrow}

\title[Gaps in Jin's Grothendieck-Group Argument]
{ Grothendieck Group of the Level
\texorpdfstring{$-1$}{-1} Type \texorpdfstring{$D$}{D}:\\ Gaps in Jin's Proof and a Complete Replacement Proof}

\author{Morris Jain}


\begin{document}

\begin{abstract}
Jin recently claims that the distinguished level
$-1$ vacuum block of $L_{-1}(D_\ell)$ realizes the specialized dual
affine left-cell module attached to the subregular cell containing
$s_0$.  We show that the proof of this Grothendieck-group statement
contains a  gap.  We give an
explicit counterexample over the discrete valuation ring
$\Bbbk[[t]]$, showing that the categorical implication used in the
proof is false in general.  Consequently, Jin's Theorem~6.3 does not
follow from the argument given in his Section~6.2.

We then supply a complete replacement proof for the
grading-restricted level-$-1$ vacuum block.  Combining the classification of simple objects, we prove finite length, construct
an injective signed normalized-character map into a completed
singular orbit module, and identify its image with the specialized
dual subregular left-cell module.  Thus the cell-module realization
claimed is established for the grading-restricted vacuum
block.
\end{abstract}

\maketitle

\section{Introduction}

Shan, Yan, and Zhao conjectured that distinguished blocks of modules
over certain simple affine vertex algebras carry dual affine left-cell
representations \cite[Conjecture~5.2.1]{SYZ2026}.  For the simple
affine vertex algebra $L_{-1}(D_\ell)$, $\ell\geq5$, the relevant
affine Weyl-group element is the affine simple reflection $s_0$, and
the predicted cell is the subregular left cell $\CellL(s_0)$.

Jin's Theorem~6.2 in \cite{Jin2026} gives the following
classification.  In the notation used below,
\begin{equation}\label{eq:intro-simple-classification}
  \Irr\cO_{-\Lambda_0}\bigl(L_{-1}(D_\ell)\bigr)
  =
  \left\{
    L\bigl(w_i\dotact(-\Lambda_0)\bigr)
    \;\middle|\;
    0\leq i\leq\ell
  \right\},
\end{equation}
where
\begin{equation}\label{eq:intro-cell-list}
  \CellL(s_0)=\{w_0,w_1,\ldots,w_\ell\}.
\end{equation}
We take the classification
\eqref{eq:intro-simple-classification} and the cell calculation
\eqref{eq:intro-cell-list} as inputs.  We do not revisit the finite
$W$-algebra argument, the exhaustion of possible highest weights, or
the identification of the candidate quotient with
$L_{-1}(D_\ell)$.  

In Section~6.2 of \cite{Jin2026}, Jin invokes an ambient realization
and writes an embedding
\begin{equation}\label{eq:intro-claimed-embedding}
  K_0\cO_{-\Lambda_0}\bigl(L_{-1}(D_\ell)\bigr)
  \longhookrightarrow
  \left.
  \left(
    \cH_{\aff}^{\dual}/
    \cH_{\aff,\not\leq_L s_0}^{\dual}
  \right)
  \right|_{q=1},
\end{equation}
which is asserted to send
$[L(y\dotact(-\Lambda_0))]$ to $D_y|_{q=1}$.  Since the labels of the
simple objects are the elements of \eqref{eq:intro-cell-list}, the
proof then identifies the image of the entire left-hand side with the
span of the corresponding $D_y$.

For the purpose of the present paper, we grant
\eqref{eq:intro-claimed-embedding}, including its injectivity and its
asserted values on the classes of simple modules.  Our criticism is
therefore independent of the construction of the ambient realization.
Even under this assumption, Theorem~6.3 does not follow from
Theorem~6.2.  The latter classifies simple objects, but it does not
imply that every object has finite length or that every class in the
ordinary Grothendieck group is a finite integral combination of
simple classes.  That conclusion requires a finite-length theorem or
some other d\'evissage statement.

 It is
false in general.  For the discrete valuation ring
$R=\Bbbk[[t]]$, the category of finitely generated $R$-modules has
exactly one simple object, but the class of that simple object does
not generate its Grothendieck group.  The same example also shows
that a full exact inclusion, even a Serre inclusion, need not induce
an injection on Grothendieck groups.  

The final stability statement in Jin's proof does not repair the
problem.  The image is declared $\whW$-stable only after it has been
identified with the image of the dual left-cell module.  Thus the
stability conclusion depends on the same image equality whose proof
requires the missing d\'evissage step; it cannot serve as an
independent justification of that equality.

A correct proof also has to keep separate the ordinary exact
Grothendieck group from the completed character module in which
affine inverse Kazhdan--Lusztig expansions naturally live.  In
addition, for the finite-length argument below we retain the degree
direction and work with the extended affine Cartan, so that the
relevant generalized weight spaces are finite dimensional.  Finally,
we fix the square-root specialization explicitly and track the parity
factor forced by the $*$-twisted dual Hecke action.  

Let
\begin{equation}\label{eq:intro-C-ell}
  \cC_\ell
  :=
  \cO_{-\Lambda_0}^{\gr}
  \bigl(L_{-1}(D_\ell)\bigr)
\end{equation}
be the grading-restricted block defined in
\cref{sec:precise-block}, and put
\begin{equation}\label{eq:intro-Li}
  L_i:=L\bigl(w_i\dotact(-\Lambda_0)\bigr),
  \qquad 0\leq i\leq\ell.
\end{equation}
We prove that every object of $\cC_\ell$ has finite length.  Hence
\begin{equation}\label{eq:intro-K0-basis}
  K_0(\cC_\ell)
  =
  \bigoplus_{i=0}^{\ell}\Z[L_i].
\end{equation}
We then define the signed normalized character map by
\begin{equation}\label{eq:intro-signed-character}
  \snch([L_i])
  :=
  \eps(w_i)\,\whR\,\ch L_i.
\end{equation}
The sign in \eqref{eq:intro-signed-character} is forced by the
$*$-twisted dual Hecke action.

\begin{theorem}\label{thm:intro-main}
Let $\ell\geq5$ and assume
\eqref{eq:intro-simple-classification}.  Then $\cC_\ell$ is a length
category.  The map \eqref{eq:intro-signed-character} is an injective
homomorphism into a completed singular orbit module, and its image is
stable under the ordinary action of $\whW$ on that module.  There is
a basis-preserving $\whW$-equivariant isomorphism
\begin{equation}\label{eq:intro-main-isomorphism}
  K_0(\cC_\ell)
  \xrightarrow{\ \sim\ }
  \cH_{\aff,\CellL(s_0)}^{\dual}(1),
  \qquad
  [L_i]\longmapsto D_{w_i},
\end{equation}
where the specialization on the right is defined explicitly in
\cref{sec:group-level-dual}.
\end{theorem}

\begin{corollary}\label{cor:intro-SYZ}
For every $\ell\geq5$, the grading-restricted level-$-1$ vacuum block
of $L_{-1}(D_\ell)$ satisfies the dual affine left-cell realization
predicted by Shan--Yan--Zhao and claimed by Jin in this case.
\end{corollary}

\section{Gap in Jin's paper}
\label{sec:Jin-argument}

Grant the ambient embedding
\eqref{eq:intro-claimed-embedding} and its asserted values on simple
classes.  Jin's argument then uses Theorem~6.2 to identify the image
of the entire Grothendieck group with
\[
  \Span_\Z\{D_y:y\in\CellL(s_0)\}.
\]
This inference is false without d\'evissage.  A classification of
simple objects says nothing, by itself, about whether an arbitrary
object admits a finite Jordan--H\"older series.  Unless every object
has finite length, an ordinary Grothendieck-group class need not be
generated by the classes of simple objects.

Equivalently, the implication
\begin{equation}\label{eq:false-devissage-implication}
  \Irr(\mathscr A)=\{S_1,\ldots,S_r\}
  \quad\Longrightarrow\quad
  K_0(\mathscr A)=\sum_{i=1}^r\Z[S_i]
\end{equation}
is false for general abelian categories.    The counterexample in
\cref{prop:DVR-counterexample} shows that
\eqref{eq:false-devissage-implication} fails even for a very elementary
noetherian module category.

\section{A counterexample to the d\'evissage step and technical preliminaries}
\label{sec:counterexamples}

\subsection{The discrete valuation ring example}

\begin{proposition}\label{prop:DVR-counterexample}
Let $R=\kfield[[t]]$, where $\kfield$ is a field.  Set
\[
  \mathscr B=R\text{-}\mathrm{mod}_{\mathrm{fg}},
  \qquad
  \mathscr A=R\text{-}\mathrm{mod}_{\fl}.
\]
Then $\mathscr B$ has exactly one simple object, namely
$\kfield=R/(t)$, but $[\kfield]$ does not generate
$K_0(\mathscr B)$.  Moreover, $\mathscr A$ is a Serre subcategory of
$\mathscr B$, while the inclusion-induced homomorphism
\[
  K_0(\mathscr A)\longrightarrow K_0(\mathscr B)
\]
is not injective.
\end{proposition}

\begin{proof}
The short exact sequence
\begin{equation}\label{eq:DVR-sequence}
  0\longrightarrow R
  \xrightarrow{\,t\,}
  R\longrightarrow\kfield\longrightarrow0
\end{equation}
gives
\begin{equation}\label{eq:k-zero-ambient}
  [\kfield]=[R]-[R]=0
  \qquad\text{in }K_0(\mathscr B).
\end{equation}
The rank homomorphism
\[
  \rank_R\colon K_0(\mathscr B)\longrightarrow\Z
\]
sends $[R]$ to $1$.  Hence $[R]\neq0$, so the unique simple class
does not generate $K_0(\mathscr B)$.

The category $\mathscr A$ is closed under subobjects, quotients, and
extensions, and is therefore a Serre subcategory.  Length is additive
on short exact sequences and gives an isomorphism
\[
  K_0(\mathscr A)\xrightarrow{\ \sim\ }\Z,
  \qquad
  [M]\longmapsto\length_R(M).
\]
Thus $[\kfield]\neq0$ in $K_0(\mathscr A)$, while
\eqref{eq:k-zero-ambient} shows that its image in
$K_0(\mathscr B)$ is zero.
\end{proof}

\begin{corollary}\label{cor:false-implications}
A classification of all simple objects does not imply that their
classes generate the ordinary Grothendieck group.  Likewise, a full
exact inclusion need not induce an injection on Grothendieck groups,
even when the smaller category is a Serre subcategory.
\end{corollary}

The first conclusion is exactly the categorical implication missing
from Jin's proof of Theorem~6.3. 

\subsection{Ordinary and completed Grothendieck groups}

Affine inverse Kazhdan--Lusztig formulas have the form
\begin{equation}\label{eq:infinite-character-shape}
  \whR\,\ch L_y
  =
  \sum_x a_{x,y}e^{xA},
\end{equation}
where infinitely many coefficients may be nonzero.  The right-hand
side belongs naturally to a product
\begin{equation}\label{eq:completed-orbit-product}
  \prod_{x\in X}\Z e^{xA},
\end{equation}
not to a finite-support direct sum.  Likewise, the modified
Grothendieck groups used in affine localization are normally inverse
limits over finite truncations.

Accordingly, one must distinguish
\begin{equation}\label{eq:three-Grothendieck-objects}
  K_0(\mathscr B),
  \qquad
  K_0(\mathscr B^{\fl}),
  \qquad
  \widehat K_0(\mathscr B)
  \ \text{or a completed character group}.
\end{equation}
The first is defined using finite short exact sequences and need not
have a basis of simple classes.  The second has such a basis when
$\mathscr B^{\fl}$ is a length category.  The third is the natural
target for \eqref{eq:infinite-character-shape}. 

  We
first prove finite length inside the grading-restricted block, so that
its ordinary $K_0$ really is the finite free abelian group on simple
classes.  Only after that step do we pass to a completed character
module in which the affine inverse Kazhdan--Lusztig formula is
meaningful.

\subsection{The grading used in the replacement argument}

Our finite-length argument requires exact finite-dimensional
weight-space functors.  For this purpose we retain the degree
direction and work with
$\mathfrak h\oplus\C K\oplus\C d$, where $d=-L_0$ on the
vertex-algebra modules considered below.  The following elementary
observation explains why the degree direction is useful.

\begin{lemma}\label{lem:vacuum-infinite-finite-Cartan-weight}
Relative to $\mathfrak h\oplus\C K$, the finite-weight-$0$,
level-$-1$ weight space of $L_{-1}(D_\ell)$ is infinite dimensional.
\end{lemma}

\begin{proof}
Choose $h\in\mathfrak h$ with $(h,h)\neq0$.  For every $n\geq1$, the
vector $h(-n)\mathbf1$ has finite Cartan weight $0$ and level $-1$.
It is nonzero: if $h(-n)\mathbf1=0$, then the affine commutation
relation gives
\[
  0
  =h(n)h(-n)\mathbf1
  =-n(h,h)\mathbf1,
\]
a contradiction.  These vectors have distinct conformal degrees and
are therefore linearly independent.
\end{proof}

Thus the finite Cartan alone cannot provide the finite-dimensional
detectors needed in our proof.  This is why the replacement theorem
is formulated for the grading-restricted block in
\cref{def:graded-vacuum-block}.  

\subsection{The square-root specialization and the character sign}

Write $v=q^{1/2}$.  The coefficient ring is
$\Z[v,v^{-1}]$, while the involution used in the dual action sends
$v$ to $-v$.  Thus the shorthand ``set $q=1$'' does not by itself
record a choice of square root.  In the replacement proof we choose
$v\mapsto1$ and keep track of the resulting parity factor.

Let $W=\langle s\mid s^2=1\rangle$.  At $v=1$, the negative
Kazhdan--Lusztig basis in the convention used for the dual cell module
is
\[
  C_1^-=1,
  \qquad
  C_s^-=s-1.
\]
Let $D_1,D_s$ be the dual basis.  Under the trace identification of
the algebraic dual with the completed group algebra,
\begin{equation}\label{eq:A1-D-basis}
  D_1=1+s,
  \qquad
  D_s=s.
\end{equation}
The $*$-twisted action satisfies $s\cdot F=-sF$.  Therefore
\begin{equation}\label{eq:A1-dual-action}
  s\cdot D_s
  =-1
  =D_s-D_1.
\end{equation}
The dual cell quotient attached to $s$ kills $D_1$, so the image
$\overline{D_s}$ satisfies
\[
  s\cdot\overline{D_s}=\overline{D_s}.
\]
Thus the one-dimensional dual cell module is the trivial
representation, not the sign representation.

Now let $A$ be fixed by $s$ and define
\[
  \Pi_A(1)=e^A,
  \qquad
  \Pi_A(s)=-e^A.
\]
Then
\begin{equation}\label{eq:A1-projection}
  \Pi_A(D_1)=0,
  \qquad
  \Pi_A(D_s)=-e^A.
\end{equation}
The normalized character of the singular simple highest-weight module
is $e^A$.  Hence $D_s$ corresponds to
\[
  -e^A=\eps(s)e^A,
\]
not to $e^A$.  In general, the character-compatible formula is
\begin{equation}\label{eq:sign-required}
  D_y
  \longmapsto
  \eps(y)\,\whR\,\ch L(y\dotact\lambda).
\end{equation}

Even granting the ambient embedding used in
\cite[Section~6.2]{Jin2026}, the argument given there does not prove
Theorem~6.3.  Theorem~6.2 classifies the simple objects, but no
finite-length or other d\'evissage result is supplied from which one
could conclude that the ordinary Grothendieck group is generated by
their classes.  The implication required for that step is false in
general by \cref{prop:DVR-counterexample}.  Consequently, the claimed
equality of images, and hence the subsequent stability conclusion,
does not follow from the argument in Section~6.2.

\section{A precise grading-restricted vacuum block}
\label{sec:precise-block}

Let $\mathfrak g=D_\ell$, $\ell\geq5$, and let
\begin{equation}\label{eq:extended-affine-algebra}
  \whge
  =
  \mathfrak g[t,t^{-1}]
  \oplus\C K
  \oplus\C d,
  \qquad
  [d,xt^n]=nxt^n.
\end{equation}
Its Cartan subalgebra is
\[
  \whh=\mathfrak h\oplus\C K\oplus\C d.
\]
At level $-1$, the Sugawara operator is defined because
\[
  -1+h^\vee=2\ell-3\neq0.
\]
On an $L_{-1}(D_\ell)$-module we take $d=-L_0$, which has the
commutator required in \eqref{eq:extended-affine-algebra}; see, for
example, \cite[Chapters~6 and 12]{Kac1990}.

\begin{definition}\label{def:graded-vacuum-block}
Let
\[
  \cC_\ell
  =
  \cO_{-\Lambda_0}^{\gr}
  \bigl(L_{-1}(D_\ell)\bigr)
\]
be the category of $L_{-1}(D_\ell)$-modules $M$ whose underlying
$\whge$-module has level $-1$, belongs to the linkage block of
$-\Lambda_0$, and satisfies the following conditions:
\begin{enumerate}[label=\textup{(\roman*)}]
  \item $M$ is the direct sum of its generalized
  $\whh$-weight spaces;
  \item every generalized $\whh$-weight space is finite dimensional;
  \item for each $\mu\in\whh^*$, only finitely many weights in
  $\mu+\widehat Q_+$ occur in $M$.
\end{enumerate}
The action of $d$ is the one induced by $-L_0$.
\end{definition}

These are the grading and support conditions in
\cite[Definition~3.1]{BKK2024}.  They are inherited by submodules,
quotients, and extensions, so $\cC_\ell$ is an abelian category.  If
\[
  0\longrightarrow M'\longrightarrow M\longrightarrow M''
  \longrightarrow0
\]
is exact in $\cC_\ell$, then it is exact on every generalized
$\whh$-weight space.  Consequently, for every
$\lambda\in\whh^*$ the functor
\begin{equation}\label{eq:exact-weight-functor}
  M\longmapsto M_\lambda
\end{equation}
is exact and finite dimensional.

The highest-weight modules in
\eqref{eq:intro-simple-classification} belong to $\cC_\ell$.
Conversely, every simple object of $\cC_\ell$ is simple in the
ambient affine block: any nonzero proper submodule would again
satisfy the defining conditions and would be an object of
$\cC_\ell$.  Thus Jin's classification gives
\begin{equation}\label{eq:graded-simple-classification}
  \Irr(\cC_\ell)
  =
  \left\{
    L\bigl(w_i\dotact(-\Lambda_0)\bigr)
    \;\middle|\;
    0\leq i\leq\ell
  \right\}.
\end{equation}
If the notation in \cite{Jin2026} is intended to include a larger
ungraded category, its Grothendieck group is not addressed by the
argument below.

Retain Jin's notation
\begin{equation}\label{eq:wi-mui-again}
  w_0=s_0,
  \qquad
  w_i=z_i s_0,
  \qquad
  \mu_0=0,
  \qquad
  \mu_i=z_i\rho-\rho
  \quad(1\leq i\leq\ell).
\end{equation}
Then
\begin{equation}\label{eq:affine-highest-weights}
  \lambda_i
  :=
  w_i\dotact(-\Lambda_0)
  =
  -\Lambda_0+\mu_i.
\end{equation}
We use the same symbol $\lambda_i$ for its extension to $\whh^*$
whose highest-weight vector has $d$-eigenvalue $0$.  The next lemma
shows that this is precisely the Sugawara normalization.

\begin{lemma}\label{lem:lowest-conformal-weight-zero}
The highest-weight vector of $L(\lambda_i)$ has conformal weight $0$
for every $0\leq i\leq\ell$.
\end{lemma}

\begin{proof}
The assertion is the vacuum normalization when $i=0$.  For
$i\geq1$, the Sugawara conformal weight is
\[
  \frac{(\mu_i,\mu_i+2\rho)}{2(-1+h^\vee)}.
\]
Since $\mu_i=z_i\rho-\rho$ and $z_i$ preserves the invariant form,
\begin{align*}
  (\mu_i,\mu_i+2\rho)
  &=(z_i\rho-\rho,z_i\rho+\rho)\\
  &=\lVert z_i\rho\rVert^2-\lVert\rho\rVert^2\\
  &=0.
\end{align*}
\end{proof}

\section{Finite length and signed normalized characters}
\label{sec:finite-length}

\subsection{A finite-detector criterion}

\begin{theorem}[Finite-detector criterion]
\label{thm:finite-detector}
Let $\mathscr A$ be an abelian category, and let
\[
  F_1,\ldots,F_r\colon
  \mathscr A\longrightarrow
  \mathrm{Vect}_{\kfield}^{\mathrm{fd}}
\]
be exact functors.  Assume that every nonzero object of $\mathscr A$
has a simple subquotient and that every simple object $S$ is detected
by at least one $F_i$, in the sense that $F_i(S)\neq0$ for some $i$.
Then every object $M$ of $\mathscr A$ has finite length, and
\begin{equation}\label{eq:abstract-length-bound}
  \length(M)
  \leq
  \sum_{i=1}^r\dim F_i(M).
\end{equation}
\end{theorem}

\begin{proof}
Consider a strict finite filtration
\[
  0=M_0\subsetneq M_1\subsetneq\cdots\subsetneq M_n=M
\]
and set $Q_j=M_j/M_{j-1}$.  Each $Q_j$ is nonzero, so it has a simple
subquotient $S_j$.  Choose $i(j)$ with
$F_{i(j)}(S_j)\neq0$.  Exactness shows that
$F_{i(j)}(Q_j)\neq0$.  Therefore
\[
  1\leq\sum_{i=1}^r\dim F_i(Q_j)
\]
for every $j$.  Summing and using exactness along the filtration gives
\[
  n
  \leq
  \sum_{j=1}^n\sum_{i=1}^r\dim F_i(Q_j)
  =
  \sum_{i=1}^r\dim F_i(M).
\]
Thus all strict finite chains of subobjects of $M$ have uniformly
bounded length.  In particular, $M$ is both noetherian and artinian.
A nonzero artinian object has a simple subobject, and a nonzero
noetherian object has a simple quotient.  Induction on the bound above
now produces a finite composition series and proves
\eqref{eq:abstract-length-bound}.
\end{proof}

\begin{remark}
For a module category closed under subobjects and quotients, the
hypothesis on simple subquotients is automatic.  A nonzero cyclic
submodule is finitely generated, so the union of a chain of proper
submodules is still proper.  Zorn's lemma therefore gives a maximal
proper submodule and hence a simple quotient.
\end{remark}

\subsection{Application to the level
\texorpdfstring{$-1$}{-1} block}

\begin{theorem}\label{thm:C-ell-finite-length}
Assume Jin's classification
\begin{equation}\label{eq:classification-for-length}
  \Irr(\cC_\ell)
  =
  \{L(\lambda_0),L(\lambda_1),\ldots,L(\lambda_\ell)\}.
\end{equation}
Then every object $M\in\cC_\ell$ has finite length.  More precisely,
\begin{equation}\label{eq:C-ell-length-bound}
  \length(M)
  \leq
  \sum_{i=0}^{\ell}\dim M_{\lambda_i}.
\end{equation}
\end{theorem}

\begin{proof}
For $0\leq i\leq\ell$, set
\[
  F_i(M):=M_{\lambda_i}.
\]
These functors are exact and finite dimensional by
\eqref{eq:exact-weight-functor}.  Every nonzero object of $\cC_\ell$
has a simple subquotient.  Indeed, a nonzero cyclic submodule has a
maximal proper submodule, and both the cyclic submodule and its
quotient remain in $\cC_\ell$.  By
\eqref{eq:classification-for-length}, the resulting simple quotient
is $L(\lambda_i)$ for some $i$.  Its highest-weight line lies in the
$\lambda_i$-weight space, so $F_i(L(\lambda_i))\neq0$.  The result
follows from \cref{thm:finite-detector}.
\end{proof}

\begin{corollary}\label{cor:K0-free-on-simples}
There is a canonical identification
\begin{equation}\label{eq:K0-free-on-simples}
  K_0(\cC_\ell)
  =
  \bigoplus_{i=0}^{\ell}\Z[L_i],
  \qquad
  L_i=L(\lambda_i).
\end{equation}
\end{corollary}

\begin{proof}
This is the usual d\'evissage theorem for a length category.
\end{proof}

This corollary supplies exactly the step missing from Jin's passage
from simple labels to the ordinary Grothendieck group.  Notice that
it is proved inside the grading-restricted vertex-algebra block and
does not rely on the ordinary Grothendieck group of the ambient
affine category.

\subsection{Signed normalized characters}

We write
\begin{equation}\label{eq:affine-Weyl-denominator}
  \whR
  :=
  e^{\whrho}
  \prod_{\alpha\in\widehat\Delta_+}
  \bigl(1-e^{-\alpha}\bigr)^{\operatorname{mult}(\alpha)}
\end{equation}
for the affine Weyl denominator.  Put
\begin{equation}\label{eq:A-ell}
  A_\ell
  :=
  -\Lambda_0+\whrho
  =
  (2\ell-3)\Lambda_0+\rho.
\end{equation}
Its stabilizer is
\begin{equation}\label{eq:A-ell-stabilizer}
  W_{A_\ell}
  =
  \Stab_{\whW}(A_\ell)
  =
  \langle s_0\rangle.
\end{equation}
Let $X_{A_\ell}$ be the set of maximal-length representatives of the
right cosets $\whW/W_{A_\ell}$.  The map
\[
  X_{A_\ell}\longrightarrow\whW A_\ell,
  \qquad
  x\longmapsto xA_\ell,
\]
is a bijection.

\begin{definition}\label{def:completed-orbit-module}
The completed orbit module is
\begin{equation}\label{eq:completed-orbit-module}
  \widehat{\cX}_{A_\ell}
  :=
  \prod_{x\in X_{A_\ell}}\Z e^{xA_\ell}.
\end{equation}
The affine Weyl group acts by the ordinary action on exponents:
\begin{equation}\label{eq:ordinary-orbit-action}
  w\cdot e^\eta=e^{w\eta}.
\end{equation}
\end{definition}

The ordinary normalized character is
\[
  \nch L_i=\whR\,\ch L_i.
\]
After formal characters are applied to the singular inverse
Kazhdan--Lusztig formula, each $\nch L_i$ is an element of
$\widehat{\cX}_{A_\ell}$; the coefficient formula is proved in
\cref{thm:character-realization}.  The character-compatible Hecke
convention requires the parity factor isolated in
\eqref{eq:sign-required}.

\begin{definition}\label{def:signed-normalized-character}
Define
\begin{equation}\label{eq:signed-normalized-character-map}
  \snch\colon K_0(\cC_\ell)
  \longrightarrow
  \widehat{\cX}_{A_\ell}
\end{equation}
by
\begin{equation}\label{eq:snch-simple}
  \snch([L_i])
  :=
  \eps(w_i)\,\whR\,\ch L_i.
\end{equation}
Equivalently, for $M\in\cC_\ell$,
\begin{equation}\label{eq:snch-composition-multiplicities}
  \snch([M])
  =
  \sum_{i=0}^{\ell}
  [M:L_i]\,
  \eps(w_i)\,\whR\,\ch L_i.
\end{equation}
\end{definition}

The definition uses finite length.  It is not the ordinary character
of $M$; it is the parity-renormalized character dictated by the dual
Hecke basis.

\begin{proposition}\label{prop:snch-injective}
The homomorphism \eqref{eq:signed-normalized-character-map} is
injective.
\end{proposition}

\begin{proof}
Suppose
\[
  \sum_{i=0}^{\ell}a_i\,
  \eps(w_i)\whR\ch L_i=0
\]
with $a_i\in\Z$, and choose an index $j$ for which
$a_j\neq0$ and $w_jA_\ell$ is maximal, among those with nonzero
coefficient, in the highest-weight order.  The normalized character
of $L_j$ has leading term $e^{w_jA_\ell}$ with coefficient $1$, and
all remaining terms are lower.  The orbit points $w_iA_\ell$ are
distinct.  Hence the coefficient of $e^{w_jA_\ell}$ in the displayed
sum is $a_j\eps(w_j)$, a contradiction.  Thus the signed normalized
characters are linearly independent, and the result follows from
\cref{cor:K0-free-on-simples}.
\end{proof}

\section{The dual Hecke module and its character realization}
\label{sec:Hecke-character}

\subsection{The group-level specialization}
\label{sec:group-level-dual}

Let $(W,S)$ be a Coxeter system.  We first define the specialization
used below without relying on the ambiguous shorthand $q=1$.

Let $\Z[W]$ have standard basis $\{T_w\mid w\in W\}$.  Denote by
$\posC_w$ the specialization at $q=1$ of the positive
Kazhdan--Lusztig basis used in \cite{BKK2024}, and by $\negC_w$ the
specialization at $v=1$ of the negative basis used in
\cite{SYZ2026}.  Let
\[
  \sigma\colon\Z[W]\longrightarrow\Z[W],
  \qquad
  \sigma(T_w)=\eps(w)T_w.
\]
The two bases are related by the usual sign automorphism of the group
algebra \cite{KL1979}:
\begin{equation}\label{eq:positive-negative-basis-relation}
  \negC_w
  =
  \eps(w)\,\sigma(\posC_w).
\end{equation}
For example, in rank one,
\[
  \posC_s=1+s,
  \qquad
  \negC_s=s-1.
\]

Let $\Z[W]^*=\Hom_\Z(\Z[W],\Z)$.  Define
$D_y\in\Z[W]^*$ by
\begin{equation}\label{eq:D-dual-negative-basis}
  \langle D_y,\negC_z\rangle
  =
  \delta_{y,z^{-1}}.
\end{equation}
The $*$-twisted left $W$-action is
\begin{equation}\label{eq:group-level-dual-action}
  \langle w\cdot f,h\rangle
  =
  \eps(w)\langle f,hT_w\rangle,
  \qquad
  w\in W,
  \quad
  f\in\Z[W]^*,
  \quad
  h\in\Z[W].
\end{equation}
This is the group-level form of the dual action in
\cite[Equation~(2.1.1)]{SYZ2026} after choosing $v=1$.

The submodule
\begin{equation}\label{eq:restricted-dual-module}
  \cH_W^{\dual}(1)
  :=
  \bigoplus_{y\in W}\Z D_y
  \subset\Z[W]^*
\end{equation}
is stable under \eqref{eq:group-level-dual-action}.  It carries the
left-cell filtration inherited from the negative
Kazhdan--Lusztig basis.  For $w\in W$, put
\begin{equation}\label{eq:parabolic-dual-quotient-general}
  \cQ_w^{\dual}(1)
  :=
  \cH_W^{\dual}(1)
  \big/
  \Span_\Z\{D_y\mid y\not\leq_L w\}.
\end{equation}
The images of the $D_y$ with $y\leq_L w$ form a basis.

If $\mathbf c=\CellL(w)$, the canonical cell inclusion gives
\begin{equation}\label{eq:group-level-cell-inclusion}
  \cH_{W,\mathbf c}^{\dual}(1)
  \longhookrightarrow
  \cQ_w^{\dual}(1).
\end{equation}
This is the group-level specialization of
\cite[Equation~(2.1.2)]{SYZ2026}.

\begin{remark}\label{rem:twisted-dual-not-ordinary-contragredient}
The action in \eqref{eq:group-level-dual-action} is not the ordinary
contragredient action on the dual of a right cell module.  It differs
by the sign character.  Equivalently,
\[
  \cH_{W,\CellL(w)}^{\dual}(1)
  \cong
  \Hom_\Z\bigl(
    \cH_{W,\CellR(w^{-1})}(1),\Z
  \bigr)
  \otimes\Z_{\sgn},
\]
with the appropriate left--right convention.  This is why the
rank-one dual cell attached to $s$ is trivial.
\end{remark}

\subsection{The trace completion and inverse
Kazhdan--Lusztig coefficients}

Let
\begin{equation}\label{eq:completed-group-algebra}
  \widehat{\Z[W]}
  :=
  \prod_{x\in W}\Z T_x.
\end{equation}
The trace pairing identifies $\Z[W]^*$ with
$\widehat{\Z[W]}$ by
\begin{equation}\label{eq:trace-identification}
  \iota(f)
  :=
  \sum_{x\in W}
  \langle f,T_{x^{-1}}\rangle T_x.
\end{equation}
Under this identification, the action
\eqref{eq:group-level-dual-action} becomes
\begin{equation}\label{eq:action-in-completion}
  w\cdot F
  =
  \eps(w)T_wF.
\end{equation}
Left multiplication by a fixed group element is well defined on the
product completion.

Let $m_z^x(q)$ be the inverse Kazhdan--Lusztig polynomials in the
normalization of \cite[Appendix~A]{BKK2024}, and write
\[
  m_z^x:=m_z^x(1).
\]
Thus
\begin{equation}\label{eq:positive-inverse-KL-expansion}
  T_x
  =
  \sum_{z\leq x}
  \eps(xz^{-1})m_z^x\,\posC_z.
\end{equation}

\begin{lemma}\label{lem:D-expansion-in-standard-basis}
Under the trace identification \eqref{eq:trace-identification},
\begin{equation}\label{eq:D-standard-expansion-correct}
  \iota(D_y)
  =
  \sum_{x\in W}m_y^xT_x.
\end{equation}
\end{lemma}

\begin{proof}
Apply $\sigma$ to
\eqref{eq:positive-inverse-KL-expansion}.  By
\eqref{eq:positive-negative-basis-relation},
\begin{align*}
  \eps(x)T_x
  &=
  \sum_{z\leq x}
  \eps(xz^{-1})m_z^x\,
  \eps(z)\negC_z\\
  &=
  \eps(x)
  \sum_{z\leq x}m_z^x\negC_z.
\end{align*}
Hence
\begin{equation}\label{eq:standard-in-negative-basis}
  T_x
  =
  \sum_{z\leq x}m_z^x\negC_z.
\end{equation}
It follows from \eqref{eq:D-dual-negative-basis} that
\[
  \langle D_y,T_{x^{-1}}\rangle
  =
  m_{y^{-1}}^{x^{-1}}.
\]
The inversion symmetry of inverse Kazhdan--Lusztig polynomials gives
\[
  m_{y^{-1}}^{x^{-1}}=m_y^x.
\]
Substitution into \eqref{eq:trace-identification} proves the claim.
\end{proof}

This lemma is the source of the completion.  Although $D_y$ is one
basis vector in the restricted dual, its standard-coordinate
expansion can have infinite support.

\subsection{Projection to a singular orbit}

Let $A$ be dominant integral, and suppose its stabilizer
$W_A\subset W$ is finite.  Let $X_A$ be the set of maximal-length
representatives of the right cosets $W/W_A$.  Define
\begin{equation}\label{eq:general-completed-orbit}
  \widehat{\cX}_A
  :=
  \prod_{x\in X_A}\Z e^{xA},
\end{equation}
with the ordinary $W$-action on exponents.

\begin{definition}\label{def:singular-projection}
For
$F=\sum_{z\in W}a_zT_z\in\widehat{\Z[W]}$, define
\begin{equation}\label{eq:Pi-A-definition}
  \Pi_A(F)
  :=
  \sum_{x\in X_A}
  \left(
    \sum_{u\in W_A}
    \eps(xu)a_{xu}
  \right)e^{xA}.
\end{equation}
Equivalently, \eqref{eq:Pi-A-definition} is the formal sum
\begin{equation}\label{eq:Pi-A-formal}
  \Pi_A(F)
  =
  \sum_{z\in W}\eps(z)a_ze^{zA},
\end{equation}
with terms having the same exponent combined.
\end{definition}

The inner sum in \eqref{eq:Pi-A-definition} is finite because $W_A$
is finite, so the formal expression in \eqref{eq:Pi-A-formal} is
well defined after equal exponents are combined.

\begin{lemma}\label{lem:Pi-equivariant}
The map
\[
  \Pi_A\colon
  \widehat{\Z[W]}
  \longrightarrow
  \widehat{\cX}_A
\]
is $W$-equivariant when the source carries
\eqref{eq:action-in-completion} and the target carries the ordinary
action on exponents.
\end{lemma}

\begin{proof}
Let $F=\sum_{z\in W}a_zT_z$.  Since left multiplication by $T_w$
only permutes the standard coordinates, we may compute
coefficientwise:
\begin{align*}
  \Pi_A(w\cdot F)
  &=\sum_{z\in W}\eps(w)\eps(wz)a_z e^{wzA}\\
  &=\sum_{z\in W}\eps(z)a_z e^{wzA}\\
  &=w\cdot\Pi_A(F).
\end{align*}
For each orbit point, only the finitely many elements of one
$W_A$-coset contribute, so all three expressions are well defined.
\end{proof}

Let $w_A$ be the longest element of $W_A$.  The standard parabolic
Kazhdan--Lusztig criterion gives
\begin{equation}\label{eq:longest-left-order-criterion}
  X_A
  =
  \{y\in W\mid y\leq_L w_A\};
\end{equation}
see \cite{Deodhar1987,KT2002}.  For $x\in X_A$ and $y\in W$, set
\begin{equation}\label{eq:parabolic-inverse-coefficient}
  p_{x,y}^{A}
  :=
  \sum_{u\in W_A}\eps(u)m_y^{xu}.
\end{equation}

\begin{lemma}\label{lem:parabolic-cancellation}
For $x\in X_A$ and $y\notin X_A$, one has
\[
  p_{x,y}^{A}=0.
\]
For $x,y\in X_A$, the coefficient $p_{x,y}^{A}$ vanishes unless
$y\leq x$ in Bruhat order, and
\[
  p_{y,y}^{A}=1.
\]
\end{lemma}

\begin{proof}
Set
\[
  a_A:=\sum_{u\in W_A}\eps(u)T_u,
  \qquad
  I_A:=\Z[W]a_A.
\]
The elements of $I_A$ are exactly the elements of the group algebra
that transform by the sign character under right multiplication by
$W_A$; this may be checked on each right coset separately.

For $x\in X_A$, let $r=xw_A$ be the minimal representative of the
same right coset and put
\[
  b_x:=\eps(w_A)T_r a_A.
\]
The supports of the $b_x$ lie in distinct right cosets, and
\begin{equation}\label{eq:antispherical-standard-basis}
  b_x
  =
  T_x+
  \sum_{\substack{u\in W_A\\u\neq w_A}}
  \eps(w_Au)T_{ru},
  \qquad
  ru<x.
\end{equation}
Thus the $b_x$, $x\in X_A$, form a basis of $I_A$ that is triangular
with respect to the standard basis.

Every simple reflection $s\in W_A$ is a right descent of every
$x\in X_A$.  The Kazhdan--Lusztig multiplication formula at $v=1$
gives
\[
  \posC_xT_s=\posC_x.
\]
Applying \eqref{eq:positive-negative-basis-relation} yields
\[
  \negC_xT_s=-\negC_x,
\]
so $\negC_x\in I_A$.  Since
\[
  \negC_x=T_x+\sum_{z<x}c_{z,x}T_z,
\]
comparison with \eqref{eq:antispherical-standard-basis} and induction
on length show that
\begin{equation}\label{eq:antispherical-ideal-basis}
  I_A
  =
  \bigoplus_{x\in X_A}\Z\negC_x.
\end{equation}
This is the $v=1$ anti-spherical canonical-basis statement.

For $x\in X_A$, equations
\eqref{eq:standard-in-negative-basis} and the group-algebra
multiplication give
\begin{align*}
  T_xa_A
  &=\sum_{u\in W_A}\eps(u)T_{xu}\\
  &=\sum_{y\in W}
    \left(
      \sum_{u\in W_A}\eps(u)m_y^{xu}
    \right)\negC_y\\
  &=\sum_{y\in W}p_{x,y}^{A}\negC_y.
\end{align*}
The left-hand side belongs to $I_A$.  By
\eqref{eq:antispherical-ideal-basis}, its coefficient at $\negC_y$
must vanish whenever $y\notin X_A$.  This proves the cancellation
statement.

Now let $x,y\in X_A$.  If $m_y^{xu}\neq0$, then $y\leq xu$ in
Bruhat order.  Since $x$ is maximal in its right $W_A$-coset, one
has $xu\leq x$, and hence $y\leq x$.  If $x=y$, the term $u=1$
contributes $m_y^y=1$, while $yu<y$ for every $u\neq1$, so
$m_y^{yu}=0$.  Therefore
$p_{y,y}^{A}=1$.
\end{proof}

\begin{proposition}\label{prop:parabolic-kernel}
The restriction of $\Pi_A\circ\iota$ to
$\cH_W^{\dual}(1)$ has kernel
\begin{equation}\label{eq:parabolic-kernel}
  \Span_\Z\{D_y\mid y\not\leq_L w_A\}.
\end{equation}
It therefore induces a $W$-equivariant injection
\begin{equation}\label{eq:Theta-A-ambient}
  \Theta_A\colon
  \cQ_{w_A}^{\dual}(1)
  \longhookrightarrow
  \widehat{\cX}_A.
\end{equation}
For $y\in X_A$, suppressing the bar on the image of $D_y$ in the
quotient, one has
\begin{equation}\label{eq:Theta-D-expansion}
  \Theta_A(D_y)
  =
  \sum_{x\in X_A}
  \eps(x)p_{x,y}^{A}e^{xA}.
\end{equation}
In particular, the coefficient of $e^{yA}$ is $\eps(y)$.
\end{proposition}

\begin{proof}
By \cref{lem:D-expansion-in-standard-basis} and the definition of
$\Pi_A$,
\begin{equation}\label{eq:Pi-iota-D-expansion}
  (\Pi_A\circ\iota)(D_y)
  =
  \sum_{x\in X_A}
  \eps(x)p_{x,y}^{A}e^{xA}
  \qquad (y\in W).
\end{equation}
The first part of \cref{lem:parabolic-cancellation} shows that every
$D_y$ with $y\notin X_A$ lies in the kernel.  For $y\in X_A$, the
second part makes the matrix in
\eqref{eq:Pi-iota-D-expansion} unitriangular, with diagonal entry
$\eps(y)$, when $X_A$ is ordered by any linear extension of Bruhat
order.  Hence the images of the $D_y$, $y\in X_A$, are linearly
independent.  This proves that the kernel is exactly the span of the
remaining basis vectors.  Equation
\eqref{eq:longest-left-order-criterion} identifies that span with
\eqref{eq:parabolic-kernel}.  Equivariance follows from
\cref{lem:Pi-equivariant} and \eqref{eq:action-in-completion}.
\end{proof}

\subsection{Comparison with singular affine characters}

We now return to $W=\whW$ and
\[
  \lambda=-\Lambda_0,
  \qquad
  A=A_\ell=\lambda+\whrho,
  \qquad
  W_A=\langle s_0\rangle.
\]
For $x,y\in X_A$, abbreviate $p_{x,y}^{A}$ to $p_{x,y}$.  Then
\eqref{eq:Theta-D-expansion} reads
\begin{equation}\label{eq:Theta-D-parabolic}
  \Theta_A(D_y)
  =
  \sum_{x\in X_A}\eps(x)p_{x,y}e^{xA}.
\end{equation}

\begin{theorem}\label{thm:character-realization}
For every $y\in X_A$,
\begin{equation}\label{eq:character-realization-main}
  \Theta_A(D_y)
  =
  \eps(y)\,\whR\,\ch L(y\dotact\lambda).
\end{equation}
Thus the image of each dual canonical basis vector is the
corresponding signed normalized irreducible character.
\end{theorem}

\begin{proof}
Bezrukavnikov--Kac--Krylov label the singular block by maximal
representatives $v$ of left $W_A$-cosets and write
\[
  L_v=L(v^{-1}\dotact\lambda),
  \qquad
  M_w=M(w^{-1}\dotact\lambda).
\]
We use only the formal-character consequence of their singular
inverse Kazhdan--Lusztig formula
\cite[Proposition~3.6 and Section~3.7]{BKK2024}.  After applying the
character map and multiplying by $\whR$, that formula is the
following equality in the completed orbit module:
\begin{equation}\label{eq:BKK-singular-character-left}
  \whR\,\ch L(v^{-1}\dotact\lambda)
  =
  \sum_{w}
  \eps(wv^{-1})
  \left(
    \sum_{u\in W_A}\eps(u)m_v^{uw}
  \right)e^{w^{-1}A},
\end{equation}
where $v$ and $w$ range over maximal left-coset representatives.
This is a completed character identity, not an assertion that the
displayed infinite sum is a finite relation in the ordinary
Grothendieck group.

Set
\[
  y=v^{-1},
  \qquad
  x=w^{-1}.
\]
Then $x,y\in X_A$.  Inversion symmetry gives
\[
  m_v^{uw}=m_{v^{-1}}^{w^{-1}u^{-1}},
\]
and the substitution $u\mapsto u^{-1}$ changes the inner sum in
\eqref{eq:BKK-singular-character-left} into $p_{x,y}$.  Moreover,
\[
  \eps(wv^{-1})
  =
  \eps(x^{-1}y)
  =
  \eps(x)\eps(y).
\]
Consequently,
\begin{equation}\label{eq:BKK-right-representative-form}
  \whR\,\ch L(y\dotact\lambda)
  =
  \eps(y)
  \sum_{x\in X_A}\eps(x)p_{x,y}e^{xA}.
\end{equation}
Since $\eps(y)^2=1$, comparison with
\eqref{eq:Theta-D-parabolic} proves
\eqref{eq:character-realization-main}.
\end{proof}

\section{Restriction to the subregular cell and completion of the proof}
\label{sec:subregular-restriction}

We now specialize the preceding construction to the shifted
level-$-1$ vacuum weight
\[
  A_\ell=-\Lambda_0+\whrho.
\]
By \eqref{eq:A-ell-stabilizer}, its stabilizer is generated by $s_0$.
Thus
\[
  W_{A_\ell}=\langle s_0\rangle,
  \qquad
  w_{A_\ell}=s_0,
  \qquad
  X_{A_\ell}=\{y\in\whW\mid y\leq_Ls_0\}.
\]
The character realization of \cref{thm:character-realization}
therefore gives an injective $\whW$-homomorphism
\begin{equation}\label{eq:type-D-Theta-ambient}
  \Theta_{A_\ell}\colon
  \cQ_{s_0}^{\dual}(1)
  \longhookrightarrow
  \widehat{\cX}_{A_\ell}
\end{equation}
satisfying
\begin{equation}\label{eq:type-D-Theta-simple}
  \Theta_{A_\ell}(D_y)
  =
  \eps(y)\,\whR\,
  \ch L\bigl(y\dotact(-\Lambda_0)\bigr)
  \qquad
  (y\leq_Ls_0).
\end{equation}

The canonical dual cell inclusion
\eqref{eq:group-level-cell-inclusion} identifies
\begin{equation}\label{eq:type-D-cell-inside-quotient}
  \cH_{\aff,\CellL(s_0)}^{\dual}(1)
  \quad\text{with}\quad
  \Span_\Z\{D_y\mid y\in\CellL(s_0)\}
  \subset \cQ_{s_0}^{\dual}(1).
\end{equation}
Jin's cell calculation gives
\begin{equation}\label{eq:type-D-cell-list-final}
  \CellL(s_0)=\{w_0,w_1,\ldots,w_\ell\}.
\end{equation}
Combining \eqref{eq:type-D-Theta-simple} and
\eqref{eq:type-D-cell-list-final}, we obtain
\begin{equation}\label{eq:Theta-cell-image}
  \Theta_{A_\ell}
  \bigl(\cH_{\aff,\CellL(s_0)}^{\dual}(1)\bigr)
  =
  \bigoplus_{i=0}^{\ell}
  \Z\,\eps(w_i)\whR\ch L_i.
\end{equation}
By \cref{def:signed-normalized-character,prop:snch-injective}, the
right-hand side of \eqref{eq:Theta-cell-image} is exactly the image of
$\snch$.

\begin{theorem}[Complete Grothendieck-group theorem]
\label{thm:corrected-Jin}
Assume the simple-object classification in
\cite[Theorem~6.2]{Jin2026}.  The image of
\[
  \snch\colon
  K_0(\cC_\ell)\longrightarrow
  \widehat{\cX}_{A_\ell}
\]
is $\whW$-stable.  There is therefore a unique $\whW$-action on
$K_0(\cC_\ell)$ for which $\snch$ is equivariant.  With this action,
the map
\begin{equation}\label{eq:corrected-Jin-isomorphism}
  \Phi_\ell\colon
  K_0(\cC_\ell)
  \xrightarrow{\ \sim\ }
  \cH_{\aff,\CellL(s_0)}^{\dual}(1),
  \qquad
  [L_i]\longmapsto D_{w_i},
\end{equation}
is a basis-preserving $\whW$-equivariant isomorphism.
\end{theorem}

\begin{proof}
By \cref{cor:K0-free-on-simples}, the classes
$[L_0],\ldots,[L_\ell]$ form a basis of $K_0(\cC_\ell)$.  Equations
\eqref{eq:snch-simple} and \eqref{eq:type-D-Theta-simple} give
\[
  \snch([L_i])
  =
  \eps(w_i)\whR\ch L_i
  =
  \Theta_{A_\ell}(D_{w_i}).
\]
Thus \eqref{eq:Theta-cell-image} identifies the image of $\snch$ with
the image, under the equivariant injection $\Theta_{A_\ell}$, of the
dual left-cell module.  The image is therefore $\whW$-stable.  Since
$\snch$ is injective, the action on its image transports uniquely to
$K_0(\cC_\ell)$.  Applying $\Theta_{A_\ell}^{-1}$ on that image gives
\eqref{eq:corrected-Jin-isomorphism} and sends the displayed bases to
one another.
\end{proof}

\begin{corollary}\label{cor:SYZ-type-D-final}
For every $\ell\geq5$, the grading-restricted level-$-1$ vacuum block
of $L_{-1}(D_\ell)$ realizes the specialized dual affine left-cell
module attached to $\CellL(s_0)$.  In particular, the
cell-module realization predicted by Shan--Yan--Zhao and claimed by
Jin is valid for the grading-restricted vacuum block.
\end{corollary}

\begin{proof}
Combine \cref{thm:corrected-Jin} with
\cite[Theorem~6.2 and Proposition~2.1]{Jin2026}.
\end{proof}

\appendix

\section{Signs, inverses, and coset representatives}
\label{app:conventions}

This appendix records the convention changes used in
\cref{sec:Hecke-character}.  They are included because each of them
affects the final basis statement.

\subsection{Positive and negative Kazhdan--Lusztig bases}

Let $\posC_w$ denote the positive canonical basis used in
\cite{BKK2024} after specialization, and let $\negC_w$ denote the
negative basis used in \cite{SYZ2026}.  With
\[
  \sigma(T_w)=\eps(w)T_w,
\]
the relation is
\[
  \negC_w=\eps(w)\sigma(\posC_w).
\]
Consequently, if
\[
  T_x
  =
  \sum_{z\leq x}
  \eps(xz^{-1})m_z^x\posC_z,
\]
then
\[
  T_x
  =
  \sum_{z\leq x}m_z^x\negC_z.
\]
This cancellation is the reason that $D_y$ has the unsigned
standard-coordinate expansion in
\eqref{eq:D-standard-expansion-correct}.

\subsection{Left and right cosets}

Bezrukavnikov--Kac--Krylov use maximal representatives $v$ of left
cosets
\[
  W_A\backslash W
\]
and label the corresponding simple module by
$L(v^{-1}\dotact\lambda)$.  Jin and Shan--Yan--Zhao use the
right-coset parameter
\[
  y=v^{-1}\in W/W_A
\]
and label the module by $L(y\dotact\lambda)$.  Inversion changes
maximal left-coset representatives into maximal right-coset
representatives.  It also changes the coefficient
\[
  m_v^{uw}
\]
into
\[
  m_{v^{-1}}^{w^{-1}u^{-1}},
\]
which is the coefficient used in
\eqref{eq:parabolic-inverse-coefficient}.

\subsection{The \texorpdfstring{$*$}{star}-twisted dual action}

The dual action in \cite{SYZ2026} is defined using the involution
$v\mapsto-v$.  After choosing $v=1$, a simple reflection acts by
\[
  \langle s\cdot f,h\rangle
  =
  -\langle f,hT_s\rangle.
\]
Under the trace identification this becomes
\[
  s\cdot F=-T_sF.
\]
The singular projection includes the compensating factor $\eps(z)$:
\[
  T_z\longmapsto\eps(z)e^{zA}.
\]
It follows that the target carries the ordinary action on exponents,
while the basis vector $D_y$ maps to the signed normalized character
\[
  \eps(y)\whR\ch L(y\dotact\lambda).
\]
Omitting this sign changes the rank-one module.

\subsection{The order convention}

For a finite parabolic subgroup $W_A$ with longest element $w_A$,
the maximal right-coset representatives satisfy
\[
  X_A=\{y\in W\mid y\leq_Lw_A\}
\]
in the left-preorder convention used in \cite{SYZ2026}.  The
singular-orbit realization therefore factors through the quotient by
the span of the $D_y$ with $y\not\leq_Lw_A$.  In the present
application $W_A=\langle s_0\rangle$ and $w_A=s_0$.

\section{A useful specialization of the parabolic coefficient}
\label{app:rank-one-parabolic}

In the type-$D$ application, $W_A=\{1,s_0\}$.  Hence
\eqref{eq:parabolic-inverse-coefficient} reduces to
\begin{equation}\label{eq:type-D-parabolic-coefficient-explicit}
  p_{x,y}
  =
  m_y^x-m_y^{xs_0}.
\end{equation}
The character realization is therefore
\begin{equation}\label{eq:type-D-character-realization-explicit}
  \eps(y)\whR\ch L\bigl(y\dotact(-\Lambda_0)\bigr)
  =
  \sum_{x\in X_{A_\ell}}
  \eps(x)
  \bigl(m_y^x-m_y^{xs_0}\bigr)e^{xA_\ell}.
\end{equation}
Formula \eqref{eq:type-D-character-realization-explicit} is the
right-coset, negative-basis form of the singular inverse
Kazhdan--Lusztig character formula.

\end{document}